\documentclass[11pt]{amsart}
\usepackage[utf8]{inputenc}
\usepackage{amsmath,amssymb}
\usepackage{wrapfig}
\usepackage{url}
\usepackage{mathtools}
\usepackage{graphicx}
\usepackage{stmaryrd}
\usepackage{amsthm}
\usepackage{xcolor}
\usepackage[colorlinks=true,linkcolor=blue,citecolor=blue]{hyperref}
\usepackage[shortlabels]{enumitem}
\usepackage{comment}
\usepackage{relsize}
\usepackage{dsfont}
\usepackage{stmaryrd}
\usepackage{geometry}
\usepackage{setspace}
 \usepackage{relsize}
 \usepackage{float}
\usepackage{mathrsfs} 
\usepackage{scrextend}
\usepackage{makecell}
\setcellgapes{4pt}
\deffootnote{0em}{1.6em}{\enskip}
\allowdisplaybreaks

\newtheorem{theorem}{Theorem}[section]
\newtheorem{corollary}[theorem]{Corollary}

\newtheorem{lemma}[theorem]{Lemma}
\theoremstyle{definition}

\theoremstyle{definition}

\DeclareMathOperator{\C}{\mathbb{C}}
\DeclareMathOperator{\R}{\mathbb{R}}

\DeclareMathOperator{\N}{\mathbb{N}}

\newcommand{\sph}{\mathbb{S}^{{\it m}-1}}

\usepackage{mathtools}
\newcommand{\defeq}{\vcentcolon=}
\newcommand{\eqdef}{=\vcentcolon}
\date{\today}

\title[Lieb-Thirring inequalities on the spheres and $SO(3)$]{Lieb-Thirring inequalities on the spheres and $SO(3)$}
\author{André Kowacs}
\address[André Kowacs]{Department of Mathematics,
 Universidade Federal do Paraná (UFPR),\endgraf 
 Curitiba, Paraná, Brazil \endgraf
 and\endgraf
 Department of Mathematics: Analysis, Logic and Discrete Mathematics,
  \endgraf
  Ghent University, Belgium}
\email{andrekowacs@gmail.com}
\author{Michael Ruzhansky}
\address[Michael Ruzhansky]{
  Department of Mathematics: Analysis, Logic and Discrete Mathematics,
  \endgraf
  Ghent University, Belgium \endgraf
  and\endgraf
  Queen Mary University of London, 
   United Kingdom}
  
\email{Michael.Ruzhansky@ugent.be}
\thanks{The authors are supported  by the FWO  Odysseus  1  grant  G.0H94.18N:  Analysis  and  Partial Differential Equations, by the Methusalem programme of the Ghent University Special Research Fund (BOF)
(Grant number 01M01021) and by the FWO grant G011522N. Michael Ruzhansky was also supported by EPSRC grant EP/V005529/1. Andr\'e Kowacs was supported in part by the Coordenação de Aperfeiçoamento de Pessoal de N\'ivel Superior - Brasil (CAPES) - Finance Code 001”;}



\begin{document}
 
\begin{abstract}
In this paper, we obtain new upper bounds for the Lieb-Thirring inequality on the spheres of any dimension greater than $2$. As far as we have checked, our results improve previous results found in the literature for all dimensions greater than $2$. We also prove and exhibit an explicit new upper bound for the Lieb-Thirring inequality on $SO(3)$. We also discuss these estimates in the case of general compact Lie groups. Originally developed for estimating the sums of moments of negative eigenvalues of the Schrödinger operator in $L^2(\mathbb{R}^n)$, these inequalities have applications in quantum mechanics and other fields. 
\end{abstract}

\maketitle
\section{Introduction}
The Lieb-Thirring inequalities, originally presented in \cite{LiebThirring}, provide estimates for the $\gamma$-moments of the negative eigenvalues of the Schrödinger operator
\begin{equation}\label{deltav}
\Psi = -\Delta - V
\end{equation}
in the Hilbert space $L^2(\R^n)$. These inequalities assert that there exist constants $L_{\gamma,n}>0$ such that
\begin{equation}\label{Lieb1}
\sum_{\lambda_j \leq 0} |\lambda_j|^\gamma \leq L_{\gamma,n} \int_{\mathbb{R}^n} V(x)^{\gamma + \frac{n}{2}} dx,
\end{equation}
where $V\geq 0$ is a real-valued potential which decays fast at infinity, 
and $\gamma\geq \max\{1-\frac{n}{2},0\}$. In \cite{sharpvalue} sharp values of the constants $L_{\gamma,n}$, for $\gamma>\frac{3}{2}$ and every $n\in\N$, were obtained. These are given by 
\begin{equation}\label{sharprn}
    L^{cl}_{\gamma,n} = \frac{\Gamma(\gamma+1)}{2^n\pi^{\frac{n}{2}}\Gamma(\gamma+n/2+1)}.
\end{equation}
This type of inequality is equivalent to the following type of inequality for orthonormal families of functions. There exists a constant $k_n>0$ such that, for any orthonormal family of functions $\{\psi_j\}_{j=1}^N\subset H^1(\R^n)$, and for $\rho(x)\defeq\sum_{j=1}^N|\psi_j(x)|^2$, one has
\begin{equation}\label{LiebGenereal}
    \int_{\mathbb{R}^n} \rho(x)^{\frac{n+2}{n}}dx\leq k_n\sum_{j=1}^N\|\nabla\psi_j\|^2_2.
\end{equation}
Here the best constants $k_n$ are related to $L_{\gamma,n}$ by (see \cite{application1,LiebThirring}):
\begin{equation*}
    k_n=\frac{2}{n}\left(1+\frac{n}{2}\right)^{\frac{n+2}{n}}L^{\frac{2}{n}}_{1,n}.
\end{equation*}
Naturally, these inequalities can also be studied on manifolds. For a smooth manifold $M$ of dimension $n$, let $k_{M}>0$ denote the smallest constant such that we have
\begin{equation}\label{LiebGenereal2}
    \int_{M} \rho(x)^{\frac{n+2}{n}}dx\leq k_M\sum_{j=1}^N\|\nabla\psi_j\|_{2}^2,
\end{equation}
for every family of functions $\{\psi_j\}_{j=1}^N\subset H^1(M)$, where $\rho(x)\defeq\sum_{j=1}^N|\psi_j(x)|^2$.

On a compact manifold $M$ however, the eigenvalue $0$ needs to be accounted for. This means that instead of considering the operator \eqref{deltav}, one must consider 
\begin{equation*}
    -\Delta -\Pi(V\Pi\cdot),
\end{equation*}
where $\Pi$ denotes the orthogonal projection
\begin{equation*}
    \Pi \psi(x)=\psi(x)-\frac{1}{|M|}\int_M \psi(y)dy.
\end{equation*} 
For more references in this area, we cite \cite{survey,nsphere,liebsome,hyperbolic,torus,LiebonS2, onaclass}. It is also worth mentioning that
inequality \eqref{LiebGenereal2} has diverse applications in different fields, such as in dynamical systems (\cite{application3,application2,application1}) and quantum mechanics.
 Our work was primarily inspired by \cite{torus,LiebonS2}, where the authors obtained the upper bounds $k_{\mathbb{S}^2}\leq \frac{3}{2\pi}$, 
 and $k_{\mathbb{T}^2}\leq \frac{6}{\pi^2}$. The method they used contains ideas similar to the ones first introduced by M. Rumin \cite{Rumin}, where the inequality was considered in its first formulation \eqref{Lieb1}, in a more general setting and applied to $\R^n$. The method was later developed for the Euclidean setting in \cite{frank} and \cite{Frank2}. They also considered the first formulation of the inequality, though the second formulation \eqref{LiebGenereal} is also addressed in chapters $7.4.1$ and $7.4.2$ of \cite{Frank2}. Building on their ideas, we also obtain a new upper bound for the best constant $k_{\mathbb{S}^{m-1}}$ on the $m-1$ dimensional sphere $\mathbb{S}^{m-1}$ and for $k_{SO(3)}$ on the Lie group $SO(3)$. For every $m\geq 4$ that we have checked, our results improve the bounds obtained previously by Ilyin in \cite{nsphere} and Pan in \cite{lieb4best}, see Table \ref{table1}. However, it is worth mentioning that for the the case $m=3$ corresponding to $\mathbb{S}^2$, Ilyin and Laptev give a better upper bound in \cite{liebS2best}, as one can use a different argument, only suited for $\mathbb{S}^2$. Thus, the main interest in the following theorem is the case when $m\geq 3$, yielding new constants.
 We summarize our results as follows.

\begin{theorem}\label{teomain} Let $m\geq2$.
    Let $\{\psi_j\}_{j=1}^N\subset H^1(\sph)$ be an orthonormal family such that $\int_{\sph}\psi_j(x)dx=0$, for all $j$. Let
    \begin{equation*}
        \rho(x)\defeq\sum_{j=1}^N|\psi_j(x)|^2,
    \end{equation*}
    for $x\in \sph$. Then 
    \begin{equation*}
        \int_{\sph}\rho(x)^{\frac{m+1}{m-1}}dx\leq \frac{(m+3)}{(\sigma_m(m-1)!)^{\frac{2}{m-1}}}\left(\frac{m+1}{m-1}\right)^{\frac{m+1}{m-1}}\sum_{j=1}^N\|\nabla \psi_j\|^2_2,
    \end{equation*}
     where $\sigma_m$ denotes the area of the $m-1$ sphere, i.e.: $\sigma_m=2\pi^{\frac{m}{2}}/\Gamma(\frac{m}{2})$.
\end{theorem}
Since $SU(2)\cong \mathbb{S}^3$, we obtain as a direct consequence of Theorem \ref{teomain}, with $m=4$:
\begin{corollary}\label{corosu2}
    Let $\{\psi_j\}_{j=1}^N\subset H^1(SU(2))$ be an orthonormal family such that $\int_{SU(2)}\psi_j(x)dx=0$, for all $j$. Let
    \begin{equation*}
        \rho(x)\defeq\sum_{j=1}^N|\psi_j(x)|^2,
    \end{equation*}
    for $x\in SU(2)$. Then 
    \begin{equation*}
        \int_{SU(2)}\rho(x)^{\frac{5}{3}}dx\leq  \frac{35}{18}\sqrt[3]{\frac{25}{6\pi^{4}}}\sum_{j=1}^N\|\nabla \psi_j\|^2_2,
    \end{equation*}
    where $ \frac{35}{18}\sqrt[3]{\frac{25}{6\pi^{4}}}\approx0.680026$.
\end{corollary}
We note that in Corollary \ref{corosu2} integration is taken with respect to the standard non-normalized Haar measure on $SU(2)$ (that is, $\text{vol}(SU(2))=2\pi^2$).  
\begin{theorem}\label{teoso3}
    Let $\{\psi_j\}_{j=1}^N\in H^1(SO(3))$ be an orthonormal family such that $\int_{SO(3)}\psi_j(x)dx=0$, for all $j$. Let
    \begin{equation*}
        \rho(x)\defeq\sum_{j=1}^N|\psi_j(x)|^2,
    \end{equation*}
    for $x\in SO(3)$. Then
    \begin{align*}
\int_{SO(3)}\rho(x)^{\frac{5}{3}}dx&\leq \frac{35}{2\cdot3^{\frac{2}{3}}}\sum_{j=1}^N\|\nabla \psi_j\|^2_2,
\end{align*}
where the integration is taken with respect to the normalized Haar measure on $SO(3)$, the gradient $\nabla\psi_j$ is as in \eqref{eqgraddef} and $\frac{35}{2\cdot3^{\frac{2}{3}}}\approx8.41312$. Alternatively,
\begin{align*}
\int_{SO(3)}\rho(x)^{\frac{5}{3}}dx&\leq  \frac{35}{2\cdot3^{\frac{2}{3}}\pi^{\frac{4}{3}}}\sum_{j=1}^N\|\nabla \psi_j\|^2_2,
\end{align*}
where the integration is taken with respect to the standard non-normalized measure on $SO(3)$ 
(that is $\text{vol}(SO(3))=\pi^2$), and $\frac{35}{2\cdot3^{\frac{2}{3}}\pi^{\frac{4}{3}}}\approx1.82848$. 
\end{theorem}
 
We remark the similarity between the constants for $SU(2)$ and $SO(3)$, which is to be expected given the close connection between $SU(2)$ and $SO(3)$ (see for instance \cite{RuzPseudo2010}).

Here, the spaces $H^1(\sph)$, $H^1(SU(2))$ and $H^1(SO(3))$ denote the Sobolev spaces of order $1$ on each corresponding manifold.  We note that for all values of $m\geq 4$ that we have checked, the upper bounds for $k_{\sph}$ obtained in Theorem \ref{teomain} improve the previously best known (to the authors' knowledge) upper bounds given in \cite{nsphere}.

We present a table with values for the constants obtained, and how they compare with the values obtained previously in \cite{nsphere}, \cite{liebS2best} and \cite{lieb4best}. 
     \makegapedcells
\begin{table}[H]
\begin{tabular}{||c c c c c||} 
\hline
\multicolumn{5}{|c|}{Upper bounds for the constant $k_{\mathbb{S}^{m-1}}$
}\\
 \hline
 $m$ & From Theorem \ref{teomain} & From \cite{nsphere} & From \cite{liebS2best} & From \cite{lieb4best} \\ [0.5pt]  
 \hline 
 3 & $\frac{3}{\pi}\approx0.956$ & $\approx1.777$ & $\frac{3\pi}{32}\approx 0.2945$ & -\\ [2ex] 
 \hline
 4 & $\frac{35}{18}\sqrt[3]{\frac{25}{6\pi^{4}}}\approx0.6800$  & $\approx1.645$ & - &-\\
 \hline
 5 & $\frac{3\sqrt{3/2}}{2\pi}\approx0.5847$ & $\approx1.755$ & - &$ 0.1728$ \\
 \hline
 6 & $\frac{21}{10}\left(\frac{3^{3}\cdot7^{2}}{2\cdot5^{4}\pi^{6}}\right)^{\frac{1}{5}}\approx0.5377$ & $\approx2.009$ & - &-\\ [1ex] 
 \hline
\end{tabular}
 \caption{Upper bounds for the best constant of the Lieb-Thirring inequality on $\mathbb{S}^{m-1}$.}\label{table1}
 \end{table}
 This paper is organized as follows:\\
\indent In Section \ref{mainresults} we begin by presenting the proof of Theorem \ref{teomain}. 
 We conclude the section by recalling some of the theory of Fourier analysis on compact Lie groups and then presenting the proof of Theorem \ref{teoso3}.\\
\indent In Section \ref{fremarks} we conclude with some final remarks and ideas for future projects.
\section{Results and proofs}\label{mainresults}

\begin{proof}[Proof of Theorem \ref{teomain}]
First recall (see for instance \cite{harmonics}) that the spherical harmonics $Y_{mn}^\ell$,  satisfy
\begin{equation*}
    \mathcal{L}Y_{mn}^\ell=\lambda_{mn}Y_{mn}^\ell
\end{equation*}
where $\lambda_{mn}=n(n+m-2)$, for $n\in\N_0$, $1\leq \ell\leq k_{mn}$, and $\mathcal{L}=-\Delta$ denotes the positive Laplace-Beltrami operator $\Delta=\text{div }\nabla$. These eigenvalues have multiplicity
\begin{equation*}
    k_{mn}=(2n+m-2)\frac{(n+m-3)!}{(m-2)!n!}.
\end{equation*}
Moreover, the following identity also holds, for $x\in\sph$:
\begin{equation}\label{eq1su2}
   \sum_{\ell=1}^{k_{mn}}|Y^\ell_{mn}(x)|^2=\frac{k_{mn}}{\sigma_m},
\end{equation}
where $\sigma_m$ denotes the area of the $m-1$ sphere, i.e.: $\sigma_m=2\pi^{\frac{m}{2}}/\Gamma(\frac{m}{2})$.

    Next, fix $m\geq 2$. We exclude the constant eigenfunction and introduce the following notation for labeling the eigenfunctions and corresponding eigenvalues of the positive Laplacian with a single subscript counting multiplicities:
    \begin{equation*}
        \mathcal{L}y_j=\lambda_jy_j,
    \end{equation*}
    with $\lambda_1\leq \lambda_2\leq \dots$, where 
    \begin{equation}\label{eqdefy_j}
        \{y_j\}_{j=1}^\infty=\left\{Y_{mn}^\ell,\dots\right\},
    \end{equation}
    and
    \begin{equation*}
        (\lambda_j)_{j=1}^\infty =(n(n+m-2),\dots),
    \end{equation*}
    where the eigenvalue $n(n+m-2)$ is repeated $k_{mn}$ times.
    For $E\geq 0$, define the spectral projections:
    \begin{equation*}
        P_E=\sum_{\lambda_j<E}\langle\cdot,y_j\rangle y_j
    \end{equation*}
    and
    \begin{equation*}
        P_E^\perp=\sum_{\lambda_j\geq E}\langle\cdot,y_j\rangle y_j.
    \end{equation*}
    We denote by $\dot{\mathcal{L}}$ the positive Laplacian restricted to the invariant subspace of functions orthogonal to constants, that is, 
    \begin{equation*}
        \dot{\mathcal{L}} =P_{\lambda_1}^\perp\circ\mathcal{L}\circ P_{\lambda_1}^\perp,
    \end{equation*}
    where $\lambda_1$ is the smallest non-zero eigenvalue of $\mathcal{L}$.
    Then 
    \begin{equation}\label{lapnoncte}
        \dot{\mathcal{L}} =\sum_{j=1}^\infty \lambda_j\langle\cdot,y_j\rangle y_j,
    \end{equation}
    since the spherical harmonics form an orthonormal basis of $L^2(\sph)$. Next, notice that
    \begin{align*}
        \sum_{j=1}^\infty\lambda_j a_j&=(\lambda_1-0)\sum_{j=1}^\infty a_j+(\lambda_2-\lambda_1)\sum_{j=2}^\infty a_j +\dots=\int_0^\infty\sum_{\lambda_j\geq E}a_j dE.
    \end{align*}
    Indeed, this follows from the fact that for every $n\geq 2$, the integrand on the right hand side is constant and equal to $\sum_{j=n}^{\infty}a_j$ for $\lambda_{n-1}\leq E< \lambda_{n-1}$, and equal to $\sum_{j=1}^{\infty}a_j$ for $0\leq E<\lambda_1$.
    It follows that
    \begin{align}\label{eq1}
        \dot{\mathcal{L}} &=\int_0^\infty\sum_{\lambda_j\geq E}\langle\cdot,y_j\rangle y_j dE=
        \int_0^\infty P_E^\perp dE.
    \end{align}
    Let $\Gamma$ be the finite rank orthogonal projection:
    \begin{equation*}
        \Gamma=\sum_{j=1}^N\langle\cdot,\psi_j\rangle\psi_j.
    \end{equation*}
    Note that since $\int_{\sph} \psi_j(x)dx=0$ for $1\leq j\leq N$, the $\psi_j$ are orthogonal to the constant eigenfunction and so $\mathcal{L}\psi_j=\dot{\mathcal{L}}\psi_j$, for $1\leq j\leq N$. Next, in view of \eqref{lapnoncte} and \eqref{eq1}, we have 
    \begin{align*}
        \sum_{j=1}^N\langle\mathcal{L}\psi_j,\psi_j\rangle=\sum_{j=1}^N\langle\dot{\mathcal{L}}\psi_j,\psi_j\rangle&=\sum_{j=1}^N\int_{0}^\infty\langle P_E^\perp\psi_j,\psi_j\rangle dE\\
        &=\sum_{j=1}^N\int_0^\infty \|P_E^\perp \psi_j\|_2^2dE\\
        &=\sum_{j=1}^N\int_0^\infty\int_{\sph}|P_E^\perp \psi_j(x)|^2dxdE\\
        &=\int_{\sph}\int_0^\infty\sum_{j=1}^N|P_E^\perp \psi_j(x)|^2dEdx.
    \end{align*}
\noindent  Also, from Green's identity we know that
\begin{align*}
    \sum_{j=1}^N\|\nabla\psi_j\|^2_{L^2}=\sum_{j=1}^N\langle\mathcal{L}\psi_j,\psi_j\rangle_{L^2(\sph)}.
\end{align*}
    Therefore we obtain
    \begin{align}\label{eqgradient}
\sum_{j=1}^N\|\nabla\psi_j\|^2_{L^2}=\int_{\sph}\int_0^\infty\rho_{P_E^\perp\Gamma P_E^\perp}(x)\mathop{dE}dx,
    \end{align}
    where
    \begin{equation*}
        \rho_{P_E^\perp\Gamma P_E^\perp}(x)\defeq\sum_{j=1}^N|P_E^\perp \psi_j(x)|^2.
    \end{equation*}
    Now let $B$ be a neighbourhood around $x_0\in \sph$, with measure $|B|\leq 1$, and let $\chi_B$ be the corresponding characteristic function. Then
    \begin{align}\label{eq2}
        \left(\int_B\rho(x)dx\right)^{\frac{1}{2}}&=\|\Gamma\chi_{B}\|_{HS}\notag\\
        &\leq \|\Gamma P_E\chi_B\|_{HS}+\|\Gamma P_E^\perp\chi_B\|_{HS}\notag\\
        &=\|\Gamma P_E \chi_B\|_{HS}+\left(\int_B\rho_{P_E^\perp\Gamma P_E^\perp}(x)dx\right)^{\frac{1}{2}}.
    \end{align}
    Using that $\Gamma$ is bounded ($\|\Gamma\|\leq1$) and the fact that both $\chi_B$ and $P_E$ are projections, for $y_j$ as defined in \eqref{eqdefy_j}, we find that:
    \begin{align*}
        \|\Gamma P_E \chi_B\|_{HS}^2&\leq \|P_E\chi_B\|_{HS}^2\\
&=\sum_{\lambda_j<E}\int_{\mathbb{S}^{m-1}}|y_j(x)|^2\chi_B(x)dx\\
        &=\sum_{\substack{n(n+m-2)<E\\n\in\N}}\int_{\mathbb{S}^{m-1}}\sum_{\ell=1}^{k_{mn}}|Y_{mn}^\ell(x)|^2\chi_B(x)dx\\
        &=|B|\sum_{\substack{n(n+m-2)<E\\n\in\N}}\frac{(2n+m-2)\frac{(n+m-3)!}{(m-2)!n!}}{\sigma_m},
    \end{align*}
    where we used identity \eqref{eq1su2} in the last step, and $\mathbb{N}$ denotes the set of all positive integers. Let  $n(E) = \frac{-(m-2)+\sqrt{(m-2)^2+4E}}{2}$ denote the positive root of the equation
    \begin{equation*}
        x(x+m-2)-E=0.
    \end{equation*} For $E>(m-1)$, we then have that
    \begin{align*}
        \sum_{\substack{n(n+m-2)<E\\n\in\N}}\frac{(2n+m-2)\frac{(n+m-3)!}{(m-2)!n!}}{\sigma_m}&\leq \frac{1}{\sigma_m}\sum_{\substack{n(n+m-2)\leq E\\n\in\N}}{(2n+m-2)\frac{(n+m-3)!}{(m-2)!n!}}\\
        &= \frac{1}{\sigma_m}\left(\frac{m(m+2\lfloor n(E)\rfloor-1)(m+\lfloor n(E)\rfloor -2)!}{m!\lfloor n(E)\rfloor!}-1\right)\\
        &\leq \frac{1}{\sigma_m}\left(\frac{m(m+2 \lfloor n(E)\rfloor-1)(m+ \lfloor n(E)\rfloor -2)^{m-2}}{m!}\right)\\
        &\leq\frac{1}{\sigma_m}\left((m+2n(E)-1)\frac{(m+n(E)-2)^{m-2}}{(m-1)!}\right)\\
        &= \frac{1}{\sigma_m(m-1)!}\left((m-1)-(m-2)+\sqrt{(m-2)^2+4E})\right)\\
        &\left.\times\left(\frac{2(m-2)-(m-2)+\sqrt{(m-2)^2+4E}}{2}\right)^{m-2}\right)\\
&= \frac{1}{\sigma_m(m-1)!}\left(1+\sqrt{(m-2)^2+4E})\right)\\
        &\left.\times\left(\frac{2(m-2)-(m-2)+\sqrt{(m-2)^2+4E}}{2}\right)^{m-2}\right)\\
        &= \frac{{E}^{\frac{m-1}{2}}}{\sigma_m(m-1)!}\frac{(1+\sqrt{(m-2)^2+4E})}{\sqrt{E}}\\
        &\times\left(\frac{m-2+\sqrt{(m-2)^2+4E}}{2\sqrt{E}}\right)^{m-2}.
 \end{align*}
 The proof of the first equality (on the second line) can be found in Lemma \ref{lemmap1}. 
 Since $E\mapsto \frac{(1+\sqrt{(m-2)^2+4E})}{\sqrt{E}}$ and $E\mapsto\left(\frac{m-2+\sqrt{(m-2)^2+4E}}{2\sqrt{E}}\right)^{m-2}$ are both decreasing for $E>m-1$, we have that
        \begin{align*}
        \frac{{E}^{\frac{m-1}{2}}}{\sigma_m(m-1)!}\frac{(1+\sqrt{(m-2)^2+4E})}{\sqrt{E}}&\left(\frac{m-2+\sqrt{(m-2)^2+4E}}{2\sqrt{E}}\right)^{m-2}\\
        &\leq \frac{E^{\frac{m-1}{2}}}{\sigma_m(m-1)!}\frac{\left(1+\sqrt{m^2-4m+4+4(m-1)}\right)}{\sqrt{m-1}}\\
        &\times\left(\frac{m-2+\sqrt{m^2-4m+4+4(m-1)}}{2\sqrt{m-1}}\right)^{m-2}\\
        &= \frac{{E}^{\frac{m-1}{2}}}{\sigma_m(m-1)!} (m+1) \left(m-1\right)^{\frac{m-3}{2}} \\
        &=K_m{E}^{\frac{m-1}{2}},
    \end{align*} 
    where $K_m\defeq\frac{(m+1)(m-1)^{\frac{m-3}{2}}}{\sigma_m(m-1)!}$.  On the other hand 
    \noindent $\sum_{\substack{n(n+m-2)<E\\n\in\N}}\frac{(2n+m-2)\frac{(n+m-3)!}{(m-2)!n!}}{\sigma_m}=0$ for $0<E\leq m-1$. With this in mind, let 
    \begin{equation}
        C(E)\defeq \begin{cases}
           {E}^{\frac{m-1}{2}},\,\text{ if }E>m-1,\\
            0,\,\qquad\text{ otherwise}.
        \end{cases}
    \end{equation}
    Substituting this in \eqref{eq2}, dividing the equation by $|B|$, and letting $|B|\to 0$, we obtain
    \begin{equation*}
        \rho(x_0)^{1/2}\leq K_m^{1/2}C(E)^{1/2}+\rho_{P_E^\perp\Gamma P_E^\perp}(x_0)^{1/2}, 
    \end{equation*}
    for almost every $x_0\in \sph$.
    Since $\rho_{P_E^\perp\Gamma P_E^\perp}\geq0$, this implies
    \begin{equation*}
        \rho_{P_E^\perp\Gamma P_E^\perp}(x_0)\geq (\rho(x_0)^{1/2}-K_m^{1/2}C(E)^{1/2})_+^2,
    \end{equation*}
    for almost every $x_0\in \sph$, where the expression $(x)_+$ for  denotes the function $x\mapsto (x)_+=\max\{x,0\}$ for $x\in\R$. Integration over $x_0\in \sph$ and equality \eqref{eqgradient} then imply
    \begin{equation}\label{eqgradientfinal}
         \sum_{j=1}^N\|\nabla\psi_j\|^2_{2}\geq \int_{\sph}\int_0^\infty(\rho(x)^{1/2}-K_m^{1/2}C(E)^{1/2})_+^2 dEdx.
    \end{equation}
    Therefore, what is left is to estimate this last integral. 
   We proceed as follows.
    Define
    \begin{equation}\label{Irho}
        I(\rho)\defeq\int_0^\infty(\rho^{1/2}-K_m^{1/2}C(E)^{1/2})_+^2 dE.
    \end{equation}
  For $\rho_m=\rho/K_m$, and $I_1(\rho_m)\defeq \int_0^\infty(\rho_m^{1/2}-C(E)^{1/2})_+^2 dE$, we have that
        \begin{equation}
        {I_1(\rho_m)}\cdot{K_m}={\int_0^\infty\left({{\rho}^{1/2}}\cdot{K_m^{-1/2}}-C(E)^{1/2}\right)_+^2 dE}\cdot {K_m}=I(\rho).
    \end{equation}
Then since $C(E)=0$ for $0<E<m-1$, and $E\mapsto {E}^{\frac{m-1}{2}}$ is increasing for $E>m-1$, for $\rho_m> (m-1)^{{\frac{m-1}{2}}}$ we have that
    \begin{align*}
        I_1(\rho_m)&\geq \int_0^{m-1}({\rho_m}^{1/2}-0)^2dE+\int_{m-1}^{+\infty}({\rho_m}^\frac{1}{2}-{{E}^{\frac{m-1}{4}}})_+^2dE\\ 
        &=(m-1)\rho_m+\int_{m-1}^{\rho_m^{\frac{2}{m-1}}}({\rho_m}^\frac{1}{2}-{{E}^{\frac{m-1}{4}}})^2dE\\
        &=(m-1)\rho_m+\rho_m^{\frac{m+1}{m-1}}+\frac{2 \rho_m^{\frac{m+1}{m-1}}-2 (m-1)^{\frac{m+1}{2}}}{m+1}+\frac{8 \sqrt{\rho_m} \left((m-1)^{\frac{m+3}{4}}-\rho_m^{\frac{m+3}{2m-2}}\right)}{m+3}-m \\
        &-(m-1)\rho_m.
    \end{align*}
    Therefore
    \begin{align*}
        \frac{I_1(\rho_m)}{\rho_m^{\frac{m+1}{m-1}}}&\geq \rho_m^{-\frac{m+1}{m-1}} \left(\rho_m^{\frac{m+1}{m-1}}+\frac{2 \rho_m^{\frac{m+1}{m-1}}-2 (m-1)^{\frac{m+1}{2}}}{m+1}+\frac{8 \sqrt{\rho_m} \left((m-1)^{\frac{m+3}{4}}-\rho_m^{\frac{m+3}{2m-2}}\right)}{m+3}\right)\\
        &= \frac{ \left((m-1)^2 \rho_m^{\frac{m+1}{m-1}}+8 (m+1) (m-1)^{\frac{m+3}{4}} \sqrt{\rho_m}-2 (m+3) (m-1)^{\frac{m+1}{2}}\right)}{\rho_m^{\frac{m+1}{m-1}}(m+1)(m+3)}\eqdef J(\rho_m,m).
    \end{align*}
Differentiating $J$ with respect to the first variable, we find that
\begin{equation*}
    \frac{d}{dx}J(x,m)=2 \left((m-1)^{\frac{m-1}{2}}-2 (m-1)^{\frac{m-1}{4}} \sqrt{x}\right) x^{-\frac{2 m}{m-1}}<0,
\end{equation*}
for $x>(m-1)^{\frac{m-1}{2}}$. Therefore $J(\cdot,m)$ is decreasing for all such $x$ and we conclude that 
\begin{align*}
    \inf_{\rho_m>(m-1)^{\frac{m-1}{2}}} \left\{\frac{I_1(\rho_m)}{\rho_m^{\frac{m+1}{m-1}}}\right\}\geq\lim_{x\to\infty}J(x,m)=\frac{(m-1)^2}{(m+1)(m+3)}.
\end{align*}
 On the other hand, for $0<\rho_m\leq (m-1)^{\frac{m-1}{2}}$, we have that
\begin{align*}
    \frac{I_1(\rho_m)}{\rho_m^{\frac{m+1}{m-1}}}=\frac{\displaystyle\int_0^{m-1}\rho_m\mathop{dE}+0}{\rho_m^{\frac{m+1}{m-1}}}=(m-1)\rho_m^{-\frac{2}{m-1}}\geq 1,
\end{align*}
since $x\mapsto x^{-\frac{2}{m-1}}$ is decreasing. 
Therefore 
\begin{align*}
    \frac{I_1(\rho_m)}{\rho_m^{\frac{m+1}{m-1}}}\geq \min\left\{1,\frac{(m-1)^2}{(m+1)(m+3)}\right\}=\frac{(m-1)^2}{(m+1)(m+3)}
\end{align*}
for every $\rho_m>0$. This implies
    \begin{align*}
    \frac{I(\rho)}{\rho^{\frac{m+1}{m-1}}}=\frac{I_1(\rho_m)\cdot K_m}{(\rho_m)^{\frac{m+1}{m-1}}\cdot (K_m)^{\frac{m+1}{m-1}}}\geq \frac{(m-1)^2}{(m+1)(m+3)}\cdot (K_m)^{-\frac{2}{m-1}},
    \end{align*}
for every $\rho>0$.
 Applying this inequality to  \eqref{eqgradientfinal}, we obtain
 \begin{align*}
        \sum_{j=1}^N\|\nabla \psi_j\|_2^2 &\geq \frac{(m-1)^2}{(m+1)(m+3)} (K_m)^{-\frac{2}{m-1}}\int_{\sph}\rho(x)^{\frac{m+1}{m-1}}dx\\
        &=\frac{(m-1)^2}{(m+1)(m+3)}\left(\frac{(m+1)(m-1)^{\frac{m-3}{2}}}{\sigma_m(m-1)!}\right)^{-\frac{2}{m-1}}\int_{\sph}\rho(x)^{\frac{m+1}{m-1}}dx\\
        &=\frac{(m-1)!^{\frac{2}{m-1}}\sigma_m^{\frac{2}{m-1}}(m-1)^{\frac{m+1}{m-1}}}{(m+3)(m+1)^{\frac{m+1}{m-1}}}\int_{\sph}\rho(x)^{\frac{m+1}{m-1}}dx,
 \end{align*}
 which concludes the proof.
\end{proof}

\begin{lemma}\label{lemmap1}
    Let $m, n\in\N$, $m\geq2$. Then 
    \begin{equation}\label{eqap}
        \sum_{\ell=1}^{n}{(2\ell+m-2)\frac{(\ell+m-3)!}{(m-2)!\ell!}}=\frac{m(m+2n-1)(m+n-2)!}{m!n!}-1.
    \end{equation}
\end{lemma}
\begin{proof}
    We shall prove this by induction on $n\in\N$. Indeed, notice that for $n=1$, the right-hand side of the equality \eqref{eqap} is equal to 
    \begin{equation*}
        \frac{m(m+1)(m-1)!}{m!}-1=m,
    \end{equation*}
    while the left-hand side is given by
    \begin{equation*}
        (2+m-2)\frac{(1+m-3)!}{(m-2)!}=m.
    \end{equation*}
    This proves the case $n=1$. Now suppose the equality \eqref{eqap} holds for $n-1\in\N$. Then
    \begin{align*}
        \sum_{\ell=1}^{n}{(2\ell+m-2)\frac{(\ell+m-3)!}{(m-2)!\ell!}}&= \sum_{\ell=1}^{n-1}{(2\ell+m-2)\frac{(\ell+m-3)!}{(m-2)!\ell!}}+(2n+m-2)\frac{(n+m-3)!}{(m-2)!n!}\\
        &=\frac{m(m+2(n-1)-1)(m+(n-1)-2)!}{m!(n-1)!}-1\\
        &+(2n+m-2)\frac{(n+m-3)!}{(m-2)!n!}\\
        &=\frac{mn(m+2n-3)(m+n-3)!}{m!(n!}+(2n+m-2)\frac{(n+m-3)!}{(m-2)!n!}-1\\
        &=\frac{m(n+m-3)!(n(m+2n-3)+(m-1)(2n+m-2))}{m!n!}-1\\
        &=\frac{m(m+2n-1)(m+n-2)!}{m!n!}-1,
    \end{align*}
    which proves the formula holds for $n\in\N$ and finishes the proof.
\end{proof}

\subsection{The case of $SO(3)$}
First, we recollect some of the theory of Fourier analysis on compact Lie groups. A more detailed exposition of these subjects can be found in \cite{RuzPseudo2010}. 
Let $\text{Rep}({G})$ denote the set of all unitary irreducible representations on a compact Lie group $G$. Define $\xi,\eta\in\text{Rep}(G)$ to be equivalent, if there exists a linear bijection $A$ such that $\xi(g)\circ A= A\circ \eta(g)$, for every $g\in G$. We denote by $\widehat{G}$ the set of all equivalence classes of $\text{Rep}(G)$. Since $G$ is compact, every $\xi\in\text{Rep}(G)$ is finite dimensional, and we denote its dimension by $d_\xi\defeq\dim\xi$. Also, we shall always choose a matrix-valued representative for each class. Recall that (see \cite{RuzPseudo2010}) for the positive Laplacian (also known as the Casimir element) $\mathcal{L}$ on a compact Lie group $G$, there exist $\nu_{\xi}\geq0$ real numbers, corresponding to every matrix-valued irreducible unitary representation $\xi\in\text{Rep}(G)$, such that
\begin{equation*}
\mathcal{L}\sqrt{d_\xi}\xi_{ij}=\nu_\xi\sqrt{d_\xi}\xi_{ij},\,\forall 1\leq i,j\leq d_\xi,
\end{equation*}
where $\xi_{ij}$ denote the functions given by the coefficients of $\xi$. If we take exactly one representative of each $\xi\in[\xi]\in\widehat{G}$, then the eigenfunctions $\sqrt{d_\xi}\xi_{ij}$ form an orthonormal basis for $L^2(G)$ (with respect to the groups's normalized Haar measure), and each eigenvalue $\nu_{\xi}$ has therefore multiplicity $d_\xi^2$. Also, 
\begin{equation}\label{liegrouprep}
    \sum_{i,j=1}^{d_\xi}|\sqrt{d_\xi}\xi_{ij}(x)|^2=d_\xi\|\xi(x)\|_{HS}^2=d_\xi^2,
\end{equation}
where $\|\cdot\|_{HS}$ denotes the Hilbert-Schmidt norm. Note that if we consider the standard non-normalized Haar measure of $G$, then the orthonormal basis above has to be replaced by $\left\{\frac{\sqrt{d_\xi}\xi_{ij}}{\text{vol}(G)}\right\}$. In the case where $G=SO(3)$, the elements of $\widehat{SO(3)}$ are denoted by ${T}^\ell$, for each $\ell\in\N_0$. These are $2\ell+1$ dimensional representations.
 Let $D_1,D_2$ and $D_3$ denote the standard basis of left-invariant vector fields on $SO(3)$.
 More precisely, expressed in the local coordinates given by the Euler angles, these are given explicitly by 
\begin{align*}
   D_1&=\cos(\psi)\frac{\partial}{\partial \theta}+\frac{\sin(\psi)}{\sin(\theta)}\frac{\partial}{\partial\phi}-\frac{\cos(\theta)}{\sin(\theta)}\sin(\psi)\frac{\partial}{\partial\psi},\\
   D_2&=-\sin(\psi)\frac{\partial}{\partial \theta}+\frac{\cos(\psi)}{\sin(\theta)}\frac{\partial}{\partial\phi}-\frac{\cos(\theta)}{\sin(\theta)}\cos(\psi)\frac{\partial}{\partial\psi},\\
   D_3&=\frac{\partial}{\partial\psi},
\end{align*}
whenever $\sin(\theta)\neq 0$.
Then the positive Laplacian on $SO(3)$ is given by
\begin{equation*}
    \mathcal{L}=-(D_1^2+D_2^2+D_3^2).
\end{equation*}
Denote by $H^1(SO(3))$ the Sobolev space of order $1$ in $SO(3)$. For $\psi\in H^1(SO(3))$, define the gradient mapping $\nabla\psi:SO(3)\to\C^3$ given by
\begin{equation}\label{eqgraddef}
    \nabla\psi(x)=(D_1\psi(x),D_2\psi(x),D_3\psi(x)).
\end{equation}
Notice that
\begin{align}
    \|\nabla\psi\|_{L^2(SO(3))}^2&=\langle D_1\psi,D_1\psi\rangle_{L^2(SO(3))}+\langle D_2\psi,D_2\psi\rangle_{L^2(SO(3))}+\langle D_3\psi,D_3\psi\rangle_{L^2(SO(3))}\notag\\
    &=\langle -D_1^2\psi,\psi\rangle_{L^2(SO(3))}+\langle -D_2^2\psi,\psi\rangle_{L^2(SO(3))}+\langle -D_3^2\psi,\psi\rangle_{L^2(SO(3))}\notag\\
    &=\langle\mathcal{L}\psi,\psi\rangle_{L^2(SO(3))},\label{grad=lapla}
\end{align}
where the expression $\langle \cdot,\cdot\rangle_{L^2(SO(3))}$ denotes the usual $L^2$ inner product on $SO(3)$.

Excluding the trivial representation, we obtain the corresponding orthonormal set of non-constant eigenfunctions
 \begin{equation*}
        \{y_j\}_{j=1}^\infty=\left\{\sqrt{2\ell+1}T^\ell_{mn},\,-\ell\leq m,n\leq \ell,\,\ell\in\N\right\},
    \end{equation*}
    and their corresponding eigenvalues are given by 
    \begin{equation*}
       \{\lambda_j\}_{j=1}^\infty=\{\ell(\ell+1), \ell\in\N\}.
    \end{equation*}
    Each eigenvalue has multiplicity $(2\ell+1)^2$, since each coefficient of the matrix $T^\ell$ corresponds to a different eigenfunction. 

    We are now ready to present our final proof.

\begin{proof}
First let us fix the measure on $SO(3)$ as the standard normalized Haar measure. We proceed similarly to the proof of the case $\sph$, with the same notation, except now we obtain
    \begin{align*}
        \|\Gamma P_E \chi_B\|_{HS}^2&\leq \|P_E\chi_B\|_{HS}^2\\
&=\sum_{\lambda_j<E}\int_{SO(3)}|y_j(x)|^2\chi_B(x)dx\\
        &=\sum_{\substack{\ell(\ell+1)<E\\\ell\in\N}}\int_{SO(3)}\sum_{m,n=-\ell}^{\ell}|\sqrt{2\ell+1}T^\ell_{mn}(x)|^2\chi_B(x)dx\\
        &=|B|\sum_{{\substack{\ell(\ell+1)<E\\\ell\in\N}}}(2\ell+1)^2,
    \end{align*}
    where we used identity \eqref{liegrouprep} in the last step. Let $n(E)=\frac{-1+\sqrt{1+4E}}{2}$ be the positive root of the equation 
    \begin{equation*}
        x(x+1)-E=0.
    \end{equation*}
    Then, for $E>1(1+1)=2$ we have that
    \begin{align*}
        \sum_{{\substack{\ell(\ell+1)<E\\\ell\in\N}}}(2\ell+1)^2&\leq \sum_{{\substack{\ell(\ell+1)\leq E\\\ell\in\N}}}(2\ell+1)^2\\
        &=\frac{1}{3}(4\lfloor n(E)\rfloor^3+12\lfloor n(E)\rfloor^2+11\lfloor n(E)\rfloor)\\
        &\leq \frac{1}{3}(4 n(E)^3+12 n(E)^2+11 n(E))\\
        &=\scalebox{0.999}{$\frac{1}{3}\left(4\left(\frac{-1+\sqrt{1+4E}}{2}\right)^3+12\left(\frac{-1+\sqrt{1+4E}}{2}\right)^2+11\left(\frac{-1+\sqrt{1+4E}}{2}\right)\right)$}\\
        &=\scalebox{0.999}{$\frac{E^{\frac{3}{2}}}{3}\left(4\left(\frac{-1+\sqrt{1+4E}}{2\sqrt{E}}\right)^3+12\left(\frac{-1+\sqrt{1+4E}}{2E^{\frac{3}{4}}}\right)^2+11\left(\frac{-1+\sqrt{1+4E}}{2E^{\frac{3}{2}}}\right)\right).$}
   \end{align*}
   The proof of the first equality (on the second line) can be found in Lemma \ref{lemma2}.
   By differentiation, we see that the expression between the big parenthesis is decreasing for all $E\geq2$, hence it attains its maximum  at $E=2$. Therefore
     \begin{align*}
         \sum_{{\substack{\ell(\ell+1)<E\\\ell\in\N}}}(2\ell+1)^2&\leq  \frac{E^{\frac{3}{2}}}{3}\left(4\left(\frac{-1+3}{2\sqrt{2}}\right)^3+12\left(\frac{-1+3}{2^{\frac{7}{4}}}\right)^2+11\left(\frac{-1+3}{2^{\frac{5}{2}}}\right)\right)\\
        &=\frac{E^{\frac{3}{2}}}{3}\left({\sqrt{2}}+{3}{\sqrt{2}}+\frac{11}{{2\sqrt{2}}}\right)\\
        &=\frac{9\sqrt{2}E^{\frac{3}{2}}}{4}=KE^{\frac{3}{2}},
    \end{align*}
    where $K\defeq\frac{9\sqrt{2}}{4}$.
     On the other hand $ \sum_{\substack{n(n+1)<E\\n\in\N}}(2n+1)^2=0$ for $0<E\leq2$. With this in mind, let 
    \begin{equation}
        C(E)\defeq \begin{cases}
           {E}^{\frac{3}{2}},\,\quad\text{ if }E>2,\\
            0,\,\qquad\text{ otherwise}.
        \end{cases}
    \end{equation}
    Substituting this in the analogue of \eqref{eq2}, dividing the equation by $|B|$, and letting $|B|\to 0$, we obtain
    \begin{equation*}
        \rho(x_0)^{1/2}\leq K^{1/2}C(E)^{1/2}+\rho_{P_E^\perp\Gamma P_E^\perp}(x_0)^{1/2}, 
    \end{equation*}
    for almost every $x_0\in SO(3)$.
    Since $\rho_{P_E^\perp\Gamma P_E^\perp}\geq0$, this implies
    \begin{equation*}
        \rho_{P_E^\perp\Gamma P_E^\perp}(x_0)\geq (\rho(x_0)^{1/2}-K^{1/2}C(E)^{1/2})_+^2,
    \end{equation*}
    for almost every $x_0\in SO(3)$, where the expression $(x)_+$ for  denotes the function $x\mapsto (x)_+=\max\{x,0\}$ for $x\in\R$. Integration over $x_0\in SO(3)$ and equality \eqref{grad=lapla} then imply
    \begin{equation}\label{eqgradientfinalso3}
\sum_{j=1}^N\|\nabla\psi_j\|^2_{L^2}\geq \int_{SO(3)}\int_0^\infty(\rho(x)^{1/2}-K^{1/2}C(E)^{1/2})_+^2 dEdx.
    \end{equation}
    Therefore, what is left is to estimate this last integral. 
   We proceed as follows.
    Define
    \begin{equation}\label{Irho2}
        I(\rho)\defeq\int_0^\infty(\rho^{1/2}-K^{1/2}C(E)^{1/2})_+^2 dE.
    \end{equation}
  For $\rho_0=\rho/K$ and $I_0(\rho_0)=\int_0^\infty({\rho_0}^{1/2}-C(E)^{1/2})_+^2 dE$, we have that
        \begin{equation}
        I_0(\rho_0)\cdot K=\int_0^\infty({\rho}^{1/2}\cdot K^{-1/2}-C(E)^{1/2})_+^2 dE\cdot K=I(\rho).
    \end{equation}
Then since $C(E)=0$ for $0<E\leq 2$, and $E\mapsto {E}^{\frac{3}{2}}$ is increasing for $E>2$, for $\rho_0> 2^{{\frac{3}{2}}}$ we have that
    \begin{align*}
        I_0(\rho_0)&\geq \int_0^{2}({\rho_0}^{1/2}-0)^2dE+\int_{2}^{+\infty}({\rho_0}^\frac{1}{2}-{{E}^{\frac{3}{4}}})_+^2dE\\ 
        &=2\rho_0+\int_{2}^{\rho_0^{\frac{2}{3}}}({\rho_0}^\frac{1}{2}-{{E}^{\frac{3}{4}}})^2dE\\
        &=2\rho_0+\frac{1}{35} \left(9 \rho_0^{5/3}-70 \rho_0+80\cdot 2^{3/4} \sqrt{\rho_0}-56 \sqrt{2}\right)\\
        &=\frac{1}{35}\left(9 \rho_0^{5/3}+80\cdot 2^{3/4} \sqrt{\rho_0}-56 \sqrt{2}\right),
    \end{align*}
   where on the third line we simply perform an integration. Therefore
    \begin{align*}
        \frac{I_0(\rho_0)}{\rho_0^{\frac{5}{3}}}&\geq \frac{1}{35}\left(9+\frac{80\cdot2^{\frac{3}{4}}}{\rho_0^{\frac{7}{6}}}-\frac{56\sqrt{2}}{\rho_0^{\frac{5}{3}}}\right)\geq \frac{9}{35},\\
    \end{align*}
    for all $\rho_0>2^{\frac{3}{2}}$. On the other hand, for $0<\rho_0<2^{\frac{3}{2}}$ we have that
\begin{align*}
     \frac{I_0(\rho_0)}{\rho_0^{\frac{5}{3}}}=2\rho_0^{-\frac{2}{3}}&\geq1.
\end{align*}
Thus we conclude that 
\begin{align*}
    \frac{I(\rho)}{\rho}= \frac{I_0(\rho_0)}{\rho_0^{\frac{5}{3}}}\frac{K}{K^{\frac{5}{3}}}&\geq\min\left\{\frac{9}{35},1\right\}\left(\frac{9\sqrt{2}}{4}\right)^{-\frac{2}{3}}=\frac{9}{35}\left(\frac{4}{9\sqrt{2}}\right)^{\frac{2}{3}}=\frac{2\cdot3^{\frac{2}{3}}}{35},
\end{align*}
for all $\rho>0$. Applying this to inequality \eqref{eqgradientfinalso3}, we obtain
 \begin{align*}
        \sum_{j=1}^N\|\nabla \psi_j\|^2_2 &\geq \frac{2\cdot3^{\frac{2}{3}}}{35}\int_{SO(3)}\rho(x)^{\frac{5}{3}}dx.
 \end{align*}
 {\allowdisplaybreaks
Now let us consider the standard non-normalized Haar measure on $SO(3)$, where $\text{vol}(SO(3))=\pi^2$. Then as was already mentioned, we must use the orthonormal basis $\left\{\frac{\sqrt{2\ell+1}}{\pi^2}T^\ell_{mn},-\ell \leq m,n\leq \ell,\right.$ $\left.\ell\in\N\right\}$. Hence we may apply the same argument as before, but now}
\begin{align*}
    \sum_{\lambda_j<E}|y_j(x)|^2=\sum_{{\substack{\ell(\ell+1)<E\\\ell\in\N}}}\ \sum_{m,n=-\ell}^{\ell}\left|\frac{\sqrt{2\ell+1}}{\pi^2}T^\ell_{mn}(x)\right|^2=\sum_{{\substack{\ell(\ell+1)<E\\\ell\in\N}}}\frac{(2\ell+1)^2}{\pi^2}.
\end{align*}
Following the same steps as before, all we have to do is to replace the constant $K$ by $\frac{9\sqrt{2}}{4\pi^2}$, therefore we obtain
\begin{align*}
        \sum_{j=1}^N\|\nabla \psi_j\|^2_2 &\geq \frac{2\cdot3^{\frac{2}{3}}}{35}\pi^{\frac{4}{3}}\int_{SO(3)}\rho(x)^{\frac{5}{3}}dx,
 \end{align*}
 or equivalently 
    \begin{align*}
        \int_{SO(3)}\rho(x)^{\frac{5}{3}}dx&\leq  \frac{35}{2\cdot3^{\frac{2}{3}}\pi^{\frac{4}{3}}}\sum_{j=1}^N\|\nabla \psi_j\|^2_2,
    \end{align*}
    as claimed.
\end{proof}

\begin{lemma}\label{lemma2}
Let $m,n\in \N$, $m\geq 2$. Then
    \begin{equation*}
        \sum_{{\substack{\ell(\ell+1)\leq E\\\ell\in\N}}}(2\ell+1)^2=\frac{1}{3}(4\lfloor n(E)\rfloor^3+12\lfloor n(E)\rfloor^2+11\lfloor n(E)\rfloor),
    \end{equation*}
    where $n(E)=\frac{-1+\sqrt{1+4E}}{2}$.
\end{lemma}
\begin{proof}
    Indeed, notice that since $n(E)$ is the positive root of the equation $x(x+1)-E=0$, and $x\mapsto x(x+1)-E$ is increasing for $x>0$, we have that $\ell(\ell+1)\leq E$ whenever $\ell\leq n(E)$. Therefore it is enough to prove that
    \begin{equation}\label{eq1lema2}
        \sum_{\ell=1}^n(2\ell+1)^2=\frac{1}{3}(4n^3+12n^2+11n)
    \end{equation}
    for every $n\in\N$. Indeed, for $n=1$ we have that both sides of equation \eqref{eq1lema2} equate to $9$, which proves the result in this case. Suppose now that \eqref{eq1lema2} holds for some $n\in\N$. Then
    \begin{align*}
         \sum_{\ell=1}^{n+1}(2\ell+1)^2&= \sum_{\ell=1}^n(2\ell+1)^2+(2(n+1)+1)^2\\
         &=\frac{1}{3}(4n^3+12n^2+11n)+4n^2+12n+9\\
         &=\frac{1}{3}(4n^3+24n^2+47n+27).
    \end{align*}
    On the other hand
    \begin{align*}
        \frac{1}{3}(4(n+1)^3+12(n+1)^2+11n)&= \frac{1}{3}(4n^3+12n^2+12n+4+12n^2+24n+12+11n+11)\\
        &=\frac{1}{3}(4n^3+24n^2+47n+27),
    \end{align*}
   proving that equality \eqref{eq1lema2} holds for $n+1$ and therefore for every $n\in\N$ by induction.
\end{proof}

\section{Final remarks}\label{fremarks}

To conclude this paper, we remark that the technique used in these proofs is very general and could also be applied to other compact Lie groups. Notice that the mapping $C(E)$ defined in both proofs could be replaced by eigenvalue counting function for the Laplacian (excluding the $0$ eigenvalue) up to a constant. More precisely, now consider an arbitrary compact Lie group $G$, $\dim G=n$, and enumerate the non-constant eigenfunctions of the positive Laplacian on $G$ by $\{y_j\}_{j=1}^\infty$, with respective eigenvalues $\{\lambda_j\}_{j=1}^\infty$. For $E\geq 0$, define the spectral projections 
\setlength{\belowdisplayskip}{0pt} \setlength{\belowdisplayshortskip}{0pt}
\begin{equation*}
    P_E=\sum_{\lambda_j<E}\langle\cdot\,y_j\rangle y_j, \quad\text{ and } \quad P_E^\perp=\sum_{\lambda_j\geq E}\langle\cdot\,y_j\rangle y_j.
\end{equation*}
\setlength{\belowdisplayskip}{1pt} \setlength{\belowdisplayshortskip}{1pt}
Given an orthonormal family $\{\psi_j\}_{j=1}^N$ in $L^2(G)$ with zero average, denote by $\Gamma = \sum_{j=1}^N\langle\cdot,\psi_j\rangle\psi_j$ its corresponding orthogonal projection.
Then by the same arguments as before, for the characteristic function $\chi_B$ of a neighbourhood $B$ around a point $x_0\in G$, we have that
 \begin{align*}
        \|\Gamma P_E \chi_B\|_{HS}^2&\leq \|P_E\chi_B\|_{HS}^2\\
&=\sum_{\lambda_j<E}\int_{G}|y_j(x)|^2\chi_B(x)dx\\
        &=\sum_{\nu_\xi<E}\int_{G}\sum_{i,j=1}^{d_\xi}|\sqrt{d_\xi}\xi_{ij}(x)|^2\chi_B(x)dx\\
        &=\sum_{\nu_\xi<E}d_\xi\int_{G}\|\xi(x)\|_{HS}^2\chi_Bdx\\
        &=|B|\sum_{\nu_\xi<E}d_\xi^2\\
        &= |B|C(E),
    \end{align*}
where we have used equality \eqref{liegrouprep} and the fact that the eigenvalue $\nu_\xi$ has multiplicity $d_\xi^2$ (each eigenfunction corresponding to a coefficient of the matrix $\xi(x)$). Here $C(E)$ denotes the (non-zero)-eigenvalue counting function of the Laplacian.
Therefore  if a precise formula for the eigenvalue counting function of another compact Lie group is known, a similar proof could, in theory, be applied. Also, notice that due to  Weyl's formula for the eigenvalue counting function, we have that
    \begin{equation*}
        C(E)\sim N(E)\defeq KE^{\frac{n}{2}},
    \end{equation*} 
for some $K>0$. Then, the analogue of equation \eqref{eqgradientfinal} in this case, and the fact that $N(E)$ is increasing, imply
    \begin{align*}\label{eqgrad}
        \sum_{j=1}^N\|\nabla\psi_j\|^2_{L^2}\geq  I(\rho)\sim\int_0^{N^{-1}(\rho)}\left(\rho^{1/2}-K^{\frac{1}{2}}E^{\frac{n}{4}}\right)^2 dE.
    \end{align*}
After the change of variables $F=N(E)$, this yields
    \begin{align*}
        I(\rho)&\sim \left(\frac{2}{n}\right)K^{-\frac{2}{n}}\int_0^\rho\left(\rho^{\frac{1}{2}}-F^{\frac{1}{2}}\right)^2F^{\frac{2}{n}-1}dF\\
        &=\left(\frac{2}{n}\right)K^{-\frac{2}{n}}\left(\rho\int_0^\rho F^{\frac{2}{n}-1}dF-2\rho^{\frac{1}{2}}\int_0^\rho F^{\frac{2}{n}-\frac{1}{2}}dF+\int_0^\rho F^{\frac{2}{n}}dF\right)\\
        &=\left(\frac{2}{n}\right)K^{-\frac{2}{n}}\left(\frac{n}{2}\rho^{\frac{2}{n}+1}-\frac{2}{\frac{2}{n}+\frac{1}{2}}\rho^{\frac{2}{n}+1}+\frac{1}{\frac{2}{n}+1}\rho^{\frac{2}{n}+1}\right)\\
        &=\left(\frac{2}{n}\right)K^{-\frac{2}{n}}\left(\frac{n}{2}-\frac{4n}{4+n}+\frac{n}{2+n}\right)\rho^{\frac{n+2}{n}}\\
        &=\left(\frac{2}{n}\right)K^{-\frac{2}{n}}\frac{n^3}{2(n+2)(n+4)}\rho^{\frac{n+2}{n}}.
    \end{align*}
    And so, integration over $G$ implies that
    \begin{equation*}
       \int_G\rho(x)^{\frac{n+2}{n}}dx\lesssim \sum_{j=1}^N\|\nabla \psi_j\|^2_{L^2(G)}.
   \end{equation*}
   This shows that the exponent $\frac{n+2}{n}$ is the smallest positive exponent for $\rho$ for which this type of inequality is true. Notice also that in the case of $SO(3)$, where $\dim G=3$, this indeed coincides with the exponent $\frac{5}{3}$. 
   Finally, to compute the constant involved in the inequality above more explicitly, let $\tau_\rho$ denote the smallest non-zero eigenvalue of the Laplacian such that $C(\tau_\rho)\geq \rho$, for every $\rho\geq \lambda_1$, and $\tau_\rho=0$ otherwise. From the analogue of equation \eqref{eqgradientfinal} in this case, we obtain that
\begin{align*}
    \sum_{j=1}^N\|\nabla\psi_j\|^2_{L^2}&\geq \int_G\int_0^{\infty}\left(\rho(x)^{\frac{1}{2}}-C(E)^{\frac{1}{2}}\right)^2_+ dEdx \\
        &=\int_G\int_0^{\tau_{\rho(x)}}\left(\rho(x)^{\frac{1}{2}}-C(E)^{\frac{1}{2}}\right)^2 \frac{\rho(x)^{\frac{n+2}{n}}}{\rho(x)^{\frac{n+2}{n}}}dEdx \\
        &\geq \inf_{\rho>0}\left(\int_0^{\tau_\rho}\frac{\left(\rho^{\frac{1}{2}}-C(E)^{\frac{1}{2}}\right)^2}{\rho^{\frac{n+2}{n}}}dE\right)\int_G\rho(x)^{\frac{n+2}{n}}dx \\
        &\geq\inf_{\rho>0 }\left(\int_0^{\tau_\rho}\left(\rho^{-\frac{2}{n}} -2\rho^{-\frac{n+4}{2n}}C(E)^{\frac{1}{2}}+\rho^{-\frac{n+2}{n}}C(E)\right)dE\right)\int_G\rho(x)^{\frac{n+2}{n}}dx \\
        &= \inf_{\rho>0 }\left(
        \tau_\rho\cdot\rho^{-\frac{2}{n}}-
        \int_0^{\tau_\rho}\left(2\rho^{-\frac{n+4}{2n}}C(E)^{\frac{1}{2}}-\rho^{-\frac{n+2}{n}}C(E)\right)dE\right)\int_G\rho(x)^{\frac{n+2}{n}}dx.
    \end{align*}
The problem of finding the constant then reduces to evaluating the infimum. We have done this for the case of  $SO(3)$, but this can be in principle also done in a similar way for the situations where we have some explicit formula for the eigenvalue counting function $C(E)$, for example in the cases of $SU(3)$, $SO(4)$, etc.


\end{document}